\documentclass[letterpaper, 10 pt, conference]{ieeeconf}  %

\usepackage[T1]{fontenc}
\usepackage[utf8]{inputenc}
\usepackage{amssymb}
\usepackage{amsmath}
\usepackage{mathtools}
\usepackage{algorithm}
\usepackage{algorithmic}
\usepackage[hidelinks]{hyperref}
\usepackage[capitalize]{cleveref}
\usepackage{comment}

\usepackage{amsthm}
\usepackage{graphicx}
\usepackage{wrapfig}
\usepackage{capt-of} %
\graphicspath{{results/}}
\usepackage{xcolor}
\usepackage{mdframed}
\usepackage{cuted}
\usepackage{booktabs}
\usepackage{tikz}
\usepackage{pgfplots}
\pgfplotsset{compat=1.18}
\usepgfplotslibrary{groupplots}

\mdfdefinestyle{theorembox}{%
  backgroundcolor=blue!10,
  linecolor=blue!10,
  linewidth=0pt,
  roundcorner=2pt,
  innertopmargin=0.3\baselineskip,
  innerbottommargin=0.3\baselineskip,
  innerleftmargin=0.5em,
  innerrightmargin=0.5em,
  skipabove=0.8\baselineskip,
  skipbelow=0.8\baselineskip,
}
\surroundwithmdframed[style=theorembox]{theorem}

\newmdenv[style=theorembox]{contribbox}

\newtheorem{theorem}{Theorem}
\newtheorem{proposition}[theorem]{Proposition}
\newtheorem{remark}{Remark}
\newtheorem{lemma}[theorem]{Lemma}

\theoremstyle{definition}
\newtheorem{definition}{Definition}
\newtheorem{example}{Example}

\newcommand{\reals}{\mathbb{R}}
\newcommand{\naturals}{\mathbb{N}}
\newcommand{\norm}[1]{\Vert #1 \Vert}
\newcommand{\abs}[1]{| #1 |}

\IEEEoverridecommandlockouts                              %

\usepackage[final]{microtype}   %

\makeatletter
\g@addto@macro\normalsize{%
  \setlength\abovedisplayskip{3pt plus 2pt minus 1pt}%
  \setlength\belowdisplayskip{3pt plus 2pt minus 1pt}%
  \setlength\abovedisplayshortskip{0pt plus 1pt}%
  \setlength\belowdisplayshortskip{2pt plus 1pt}%
}
\makeatother
\let\labelindent\relax
\usepackage{enumitem}
\setlist{nosep,leftmargin=1.2em}

\title{\LARGE \bf
Strategically Robust Linear Quadratic Dynamic Games
}

\author{Boris Velasevic, Nicolas Lanzetti, and Eric Mazumdar%
\thanks{%
NL is supported by the PIMCO Fellows Program, by the National Science Foundation under Grant~\#CCF-2326609, and by the Resnick Sustainability Institute.
EM is supported by National Science Foundation Award 2240110.
}%
\thanks{%
All authors are with the Department of Computing and Mathematical Sciences, California Institute of Technology, Pasadena, USA, {\tt\small \{bvelasev,lnicolas,mazumdar\}@caltech.edu}.}%
}

\usepackage{eso-pic}
\AddToShipoutPictureFG*{\AtPageLowerLeft{\put(54,14){\parbox[b]{\dimexpr\paperwidth-108pt\relax}{\fontsize{7}{8}\selectfont \textcopyright\ 2026 IEEE. Personal use of this material is permitted. Permission from IEEE must be obtained for all other uses, in any current or future media, including reprinting/republishing this material for advertising or promotional purposes, creating new collective works, for resale or redistribution to servers or lists, or reuse of any copyrighted component of this work in other works.}}}\AtPageUpperLeft{\put(54,-30){\parbox[t]{\dimexpr\paperwidth-108pt\relax}{\centering\fontsize{8}{9}\selectfont Accepted for publication in the 65th IEEE Conference on Decision and Control (CDC 2026).}}}}

\begin{document}

\maketitle
\thispagestyle{empty}
\pagestyle{empty}

\begin{abstract}
We study linear quadratic dynamic games where players are uncertain about each other's control policies or goals and consequently seek to be \emph{strategically} robust. 
Building on recent work on strategically robust and risk-averse game theory, we first formalize the problem of strategically robust linear quadratic dynamic games. We show that these can be rewritten as simple transformations of linear quadratic games in which each player chooses a controller in a fictitious game in which they are faced with an adversary who is penalized for deviating from the other players' policies. This formulation naturally induces a novel notion of dynamic equilibrium, which we call a strategically robust dynamic equilibrium. We establish existence and uniqueness (in the subgame-perfect sense) of such equilibria  and show that the equilibrium policies are Markovian, linear, and can be efficiently computed via coupled backward Riccati equations. Through numerical simulations, including experiments in a network game, we illustrate the benefits of strategic robustness in designing robust and resilient decentralized control schemes.
Our experiments also expose a ``free-lunch'' phenomenon in games in which robustness does not incur
a corresponding loss in performance but can yield improvements in players’ utilities and social welfare.
\end{abstract}

\section{Introduction}

Linear quadratic dynamic games are natural models for multi-agent decision-making problems in a wide variety of control applications, including robotics~\cite{fridovich2020efficient}, aerospace~\cite{chai2020linear}, and economics~\cite{pindyck2003optimal}, to name a few. 
In these games, the state of a linear dynamical system is controlled by a set of agents that select state-feedback policies to minimize their personal control cost, which depends quadratically on the state of the system and on the agents' inputs. 
Since the agents' costs are private and do not necessarily align, this setting naturally leads to a game-theoretic formulation; see~\cite{bacsar1998dynamic}.

Across many applications of dynamic games, agents inevitably are confronted with \emph{strategic} uncertainty regarding the feedback policies of the other players and, thus, about their control inputs. 
For instance, in robotics, an agent might fear that other robots may deviate from the Nash equilibrium protocol, e.g., due to limited computation, a misspecified control objective, or irrational behavior.
Should this happen, the Nash equilibrium policy could be highly suboptimal and result in high costs. For example, \cite{nagel1995unraveling} demonstrates that strictly adhering to a standard Nash equilibrium when opponents are non-rational can lead to significantly suboptimal outcomes. This phenomenon can be exacerbated when the game is dynamic and the effect of decisions accumulates over time---as we show in our running example and simulations. 
It is thus natural for agents to hedge against strategic uncertainties.

Compared to standard uncertainty in control ``à la'' robust control~\cite{bacsar2008h}, strategic uncertainty cannot be simply modeled as an exogenous disturbance affecting the system as in e.g.,~\cite{amato2002guaranteeing,jimenez2006varepsilon,pantazis2023data}.
Indeed, strategic uncertainty is \emph{endogenous}: With their decisions, agents affect the other agents' decisions and, thus, the associated uncertainty. 
Consequently, strategic uncertainty must be handled in a ``game-theoretic'' fashion. 

To protect against strategic uncertainties, we take inspiration from recent work on strategic risk aversion in multi-agent reinforcement learning~\cite{mazumdar2025tractable} and strategically robust equilibria in one-shot games~\cite{lanzetti2025strategically,feik2026strategically} from which we also inherit the name ``strategic robustness''.
Specifically, we propose that agents choose their strategy not by playing against the policies of the other agents, but instead by reasoning over the strategy of a fictitious adversary. This adversary seeks to inflict maximum damage but is penalized for deviating too much from the policy of the other agents. As the penalty tends to infinity, the adversarial policy converges to the other players’ policy, recovering the standard Nash game in the limit. When, instead, this penalty is decreased, the fictitious agent gains power and can ultimately drive the control cost to $+\infty$---a phenomenon reminiscent of standard robust control. 
Thus, with this penalty, agents can directly tune their level of strategic robustness to interpolate between an adversarially robust controller and a game theoretic solution.

This paper instantiates this model in the setting of finite-horizon discrete-time linear quadratic dynamic games, whereby the dynamics associated with the state of the underlying system are linear and each agent's cost is quadratic. For this class of games, our contributions are as follows: 

\vspace{-1ex}\begin{contribbox}
\textbf{Contributions.}
We introduce the concept of strategic robustness in  linear quadratic dynamic games as a key feature which allows agents to hedge against strategic uncertainty, i.e., uncertainty over the behavior of the other agents. We show that this gives rise to a new equilibrium concept in such games: \emph{strategically robust dynamic equilibria}.
Under standard conditions, we establish existence and uniqueness of these equilibria, and show that they can be computed efficiently via coupled Riccati equations. 
Through numerical examples, we demonstrate that strategic robustness protects agents against misspecified policies and adversarial perturbations to other agents' controllers. Our experiments also empirically expose a ``free-lunch'' phenomenon that arises in certain classes of games in which robustness does not incur a loss in performance but can yield strict improvements in players' utilities and social welfare. Altogether, our results open the door to new directions in robust dynamic game theory.
\end{contribbox}

\noindent \textbf{Related work:} Unlike strategically robust equilibria~\cite{lanzetti2025strategically}, introduced in static games, we operate in a dynamic setting and thus seek feedback policies instead of static decisions. 
Also, our work does not consider mixed policies, but deterministic state-feedback policies, and therefore does not require the language of probability distributions, and relaxes the hard-constrained ambiguity set to a penalty.
Compared to risk-averse quantal equilibria~\cite{mazumdar2025tractable,zhang2026provably,zhang2025convergent}, we do not include bounded rationality and consider a standard control-theoretic setting with continuous state and action spaces. Our solution strategy, based on an increase in the number of players, is, nonetheless, inspired by that line of work.
Moreover, similarly to~\cite{lanzetti2025strategically,feik2026strategically,qu2026training}, we also observe empirically that strategic risk aversion and strategic robustness not only ensure protection, but can also coordinate agents and make agents more collaborative. 
Close to our work,~\cite{rabbani2026optimal} studies the effect of input disturbances in two-player linear quadratic games. There, player 1 compensates online for the observed disturbance, while player 2 does not best-respond; the resulting law is, as the authors note, not itself an equilibrium. Therefore, robustness is exogenous rather than strategic and has no obvious game-theoretic interpretation.
Finally, related work on robust dynamic games~\cite{jank2002disturbed,vandenbroek2003robust,jimenez2020robust,roy2025soft} studies games where the dynamics of each player is affected by an exogenous disturbance, as in standard robust control. In contrast, our perspective treats uncertainty endogenously, which leads to a different formulation that we study in this paper.

\section{Strategically robust dynamic equilibria}

We consider a discrete-time $N$-player linear quadratic game over a finite horizon $T\in\naturals_{>0}$.
The system state $x_t \in \mathbb{R}^n$ evolves according to
\begin{equation}\label{eq:system_dynamics}
    x_{t+1} = A_t x_t + \sum_{j=1}^N B_t^j u_t^j,
\end{equation}
where $u_t^i \in \mathbb{R}^{d_i}$ is the control input of player $i$, $A_t$ and $B_t^i$ are matrices of appropriate dimensions, and $x_0\in\reals^n$ is a given initial condition. 
In the standard setting of Nash equilibria, each player selects a state-feedback policy $\pi_t^i(\cdot)$ from the set of policies $\Pi^i$ (i.e., measurable functions from the past states to $\reals^{d_i}$), to minimize the cost functional 
\begin{equation}\label{eq:cost}
    \norm{x_T}^2_{Q_T^i}
    +
    \sum_{t=0}^{T-1} \norm{x_t}^2_{Q_t^i} + \norm{u_t^i}^2_{R_t^i} - \norm{u_t^{-i}}^2_{S_t^i}
\end{equation}
subject to the system dynamics~\eqref{eq:system_dynamics} and $u_t^i=\pi_t^i(x_0,\ldots,x_t)$. Here, $Q_t^i \succeq 0, R_t^i \succ 0$, and $S_t^i \succeq 0$ are cost matrices of appropriate dimensions, and we use the shorthand notation $\Vert x\Vert^2_Q=x^\top Qx$.
The term $\Vert u_t^{-i}\Vert^2_{S_t^i}$ penalizes the adversary directly, and $S_t^i=0$ recovers the standard LQ running cost. We retain it as the disturbance weight of an $\mathcal{H}_\infty$ problem which allows our setup to interpolate
between a nominal  Nash equilibrium and worst-case $\mathcal{H}_\infty$ controllers.

Unfortunately, Nash equilibria are brittle, and their performance can quickly deteriorate when the other players deviate from Nash equilibrium feedback policies---a feature concerning in practice, as we show next. 

\begin{example}[motivating]\label{example:motivating}
    Consider a two-player setting with scalar dynamics $x_{t+1} = 1.05 x_t + u^1_t + u^2_t$ and control costs $\sum_{t = 0}^{T - 1} (u^i_t)^2 + Q_T^i x_T^2$.
    Under the Nash equilibrium policies, the state evolves as $x_{t + 1} \approx 0.98x_t$ for the majority of the horizon. As the coefficient is very close to 1, small deviations in agents' policies can drastically impact the cost.
\end{example}

The fragility of Nash equilibria prompts players to seek protection against misspecified policies of the other players. 
Thus, we consider the strategically robust control cost  
\begin{multline}\label{eq:robust_cost}
    J^i(\pi^i,\pi^{-i}) = \sup_{\sigma^i\in\Sigma^i} \norm{x_T}^2_{Q^i_T} + 
    \\
    \sum_{t=0}^{T-1} \norm{x_t}^2_{Q_t^i} + \norm{u_t^i}^2_{R_t^i}
    - \norm{d_t^i}^2_{S_t^i}
    - \norm{d_t^{i}-u_t^{-i}}^2_{M_t^i}
\end{multline}
subject to the ``worst-case'' system dynamics 
$$
    x_{t+1} = A_t x_t + B_t^i u_t^i + B_t^{-i} d_t^i
$$
and the feedback policies $u_t^i = \pi^i_t(x_0,\ldots,x_t)$ for the player $i$, $u_t^{j}=\pi_t^{j}(x_0,\ldots,x_t)$ for all players $j\neq i$, and $d^{i}_t = \sigma^{i}_t(x_0,\ldots,x_t)$ for the fictitious adversary (here, $\Sigma^i$ is the set of policies, i.e., measurable functions from the past states to the input space $\reals^{\sum_{j\neq i} d_j}$, and $B_t^{-i}$ is the horizontal stacking of $B_t^j$ for $j\neq i$).
In other words, player $i$ evaluates their control cost against a fictitious adversary that, by selecting a feedback policy $\sigma^i$, aims to maximize this control cost, while not deviating too much from the other players' policies $\pi^{-i}$, as quantified by the penalty $\norm{d_t^{i}-u_t^{-i}}^2_{M_t^i}$ for $M_t^i\succ 0$.

This strategically robust cost naturally leads to the notion of equilibrium in which players are robust to each other's policies, which can be interpreted as the dynamic extension of strategically robust equilibria introduced in~\cite{lanzetti2025strategically}.

\begin{definition}[strategically robust dynamic equilibrium]
    A tuple of policies $(\bar\pi^{1},\ldots,\bar\pi^N)\in\Pi^1\times\ldots\times\Pi^N$ forms a strategically robust dynamic equilibrium if for all players $i$
    \begin{equation}\label{eq:sre}
        J^i(\bar\pi^i,\bar\pi^{-i})
        \leq 
        J^i(\pi^i,\bar\pi^{-i})
        \qquad 
        \forall \pi^i\in\Pi^i.
    \end{equation}
\end{definition}

\begin{remark}[level of robustness]
The matrix $M_t^i$, often taken as a simple diagonal matrix $m^i_t I$ for $m^i_t>0$, specifies the level of robustness. Specifically:
\begin{itemize}
    \item No robustness:
    As $M_t^i \to \infty$, the penalty for deviating is prohibitively large, forcing $d_t^i \to u_t^{-i}$ and so $\sigma^i=\pi^{-i}$. Thus, we recover the standard Nash equilibrium.
    
    \item Robust regime:
    For finite $M_t^i$, the adversary is instead allowed to explore deviations from the nominal strategies $u_t^{-i}$, effectively preparing player $i$ for adversarial or unpredictable behavior. Smaller $M_t^i$ make deviations even cheaper, protecting the player against stronger adversaries and enhancing robustness.  
    
    \item Heterogeneous level of robustness:
    By choosing $M_t^i$ as a block-diagonal matrix, i.e., $M^i_t = \text{diag}( (M^{i,j}_t)_{j \neq i} )$, player $i$ can parameterize their level of robustness to each opponent $j \neq i$. This allows modeling players that are perceived as more reliable than others.
\end{itemize}
\end{remark}

This game can also be interpreted as a two-stage dynamic game. In the first stage, all players select their policies $\pi^i$. In the second stage, the fictitious players select their adversary policies $\sigma^i$ to maximize the players' control cost, while not deviating too much from the policies $\pi^{-i}$. 
When proving our main result, we will show that this game is equivalent to a game where all players---including the adversaries---select their policies simultaneously. Besides simplifying our proof, this indicates that adversaries have no (dis)advantage in playing after observing the players' strategies.

\section{Existence, uniqueness, and computation of strategically robust dynamic equilibria}

We are now ready to characterize strategically robust dynamic equilibria: In our main result below, we study existence, uniqueness, and computation. 

\begin{theorem}[existence, uniqueness, and computation]\label{thm:sre}
    Consider the backward Riccati equation 
    \begin{equation}\label{eq:sre_riccati}
    \begin{aligned}
        P^i_t
        &= Q^i_t + K^{i\top}_t R^i_t K^i_t 
        - L^{i\top}_t S^i_t L^i_t
        \\
        &- (L^i_t - K^{-i}_t)^\top M^i_t (L^i_t - K^{-i}_t)
        \\
        &+ (A_t\!+\!B_t^iK_t^i\!+\!B_t^{-i}L_t^i)^\top P^i_{t+1} (A_t\!+\!B_t^iK_t^i\!+\!B_t^{-i}L_t^i), 
    \end{aligned}
    \end{equation}
    initialized with $P^i_T=Q_T^i$, where 
    {\setlength{\arraycolsep}{0.5pt}
    \begin{equation}\label{eq:sre_policies}
    \begin{bmatrix}
        K_t^1 \\ L_t^1 \\ \vdots \\ K_t^N \\ L_t^N
    \end{bmatrix}
    \!\!
    =
    \!\!
    \begin{bmatrix}
        H_{t}^{u^1,u^1} & H_{t}^{u^1,d^1} & \cdots & 0 & 0 \\
        H_{t}^{d^1,u^1} & H_{t}^{d^1,d^1} & \cdots & -M_t^{1,N} & 0 \\
        \vdots          & \vdots          & \vdots  & \vdots          & \vdots          \\
        0 & 0 & \cdots & H_{t}^{u^N,u^N} & H_{t}^{u^N,d^N} \\
        -M_t^{N,1} & 0 & \cdots & H_{t}^{d^N,u^N} & H_{t}^{d^N,d^N}
    \end{bmatrix}^{-1}
    \!\!
    \begin{bmatrix} 
        G_{t}^{u^1} \\ G_{t}^{d^1} \\ \vdots \\ G_{t}^{u^N} \\ G_{t}^{d^N}
    \end{bmatrix}
    \end{equation}
        }
    with the block matrices
    \begin{align*}
    &H_t^{u^i,u^i} \!=\! R_t^i + (B_t^i)^\top P_{t+1}^i B_t^i
        ,\quad
        H_t^{u^i,d^i} \!=\! (B_t^i)^\top P_{t+1}^i B_t^{-i},
        \\
        &H_t^{d^i,d^i} \!=\! M_t^i + S_t^i - (B_t^{-i})^\top P_{t+1}^i B_t^{-i},
        \\
        &H_t^{d^i,u^i} \!=\! -(B^{-i}_t)^\top P_{t+1}^i B_t^{i}
        , \quad
        G_t^{u^i} \!=\! -(B_t^i)^\top P_{t+1}^i A_t,
        \\
        &G_t^{d^i} \!=\! (B_t^{-i})^\top P_{t+1}^i A_t.
    \end{align*}
    and where $M_t^{i, j}$ is a sub-matrix of $M_t^i$ corresponding to the columns of the adversary of player $i$ associated with player $j$.
    \noindent Suppose that $H_t^{u^i,u^i} \succ 0$ and $H_t^{d^i,d^i} \succ 0$, and that the matrix $H_t$ is invertible at all times $t$, and consider the linear and Markovian feedback policies
    \begin{equation*}
        \bar\pi^i_t(x_t) 
        =
        K_t^i x_t.
    \end{equation*}
    Then, $(\bar\pi^1,\ldots,\bar\pi^N)$ is a strategically robust dynamic equilibrium of the game, and it is the unique subgame perfect strategically robust dynamic equilibrium. Moreover, the optimal adversarial feedback policy in~\eqref{eq:robust_cost} is uniquely given by $\bar\sigma_t^i(x_t)=L_t^ix_t$.
\end{theorem}

\cref{thm:sre}, which we prove in~\cref{thm:proof}, shows that strategically robust dynamic equilibria, just like standard Nash equilibria, (i) continue to be linear Markovian feedback policies, and (ii) result from (coupled) Riccati equations, so that they are both easy to implement and compute---all of which while additionally ensuring robustness to misspecified policies of the other players.
For a first illustration of strategic robustness, we revisit our motivating example.

\begin{example}[motivating, continued]
    Consider the setup of~\cref{example:motivating} extended to the framework of strategically robust agents. We set the robustness levels to $m^1_t = 1.2$ and $m^2_t \to +\infty$, i.e., player 1 protects against misspecified policies of player 2, while player 2 seeks none.
    Under the resulting strategically robust dynamic equilibrium, for the majority of the horizon, the state evolves according to $x_{t + 1} \approx 0.464 x_t$, remaining stable even under perturbations in the feedback policy of player 2. 
\end{example}

To conclude, we comment on the assumptions of~\cref{thm:sre}.
A glimpse of the proof of~\cref{thm:sre} suggests that the condition $H_t^{u^i,u^i} \succ 0$ (resp. $H_t^{d^i,d^i} \succ 0$) merely ensures that the player $i$ (resp. their adversary) does not carry the power to drive the control cost to $-\infty$ (resp. $+\infty$).
The presence of an assumption that couples the players' parameters---i.e., invertibility of $H_t$---is not surprising and appears also in standard linear quadratic games~\cite{bacsar1998dynamic}.
In our case, strategic robustness allows us to derive a sufficient condition for invertibility of $H_t$ that decomposes across players:

\begin{proposition}[invertibility of $H_t$]\label{prop:invertibility}
    Suppose that the following spectral dominance condition holds for all $i \in \{1,\ldots,N\}$: 
    \begin{align*}
        \min\{ &\lambda_{\min}(R_t^i + (B_t^{i})^{\top} P_{t + 1}^i B_t^i), \\ &\lambda_{\min}(S_t^i + M_t^i - (B_t^{-i})^\top P_{t + 1}^i B_t^{-i})\} \\ 
        & \quad > \frac{1}{2} \sum_{j \neq i} \max\{\sigma_{\max}({M_t^{i,j}}), \sigma_{\max}({M_t^{j,i}})\}.
    \end{align*}
    Then, $H_t$ is invertible. In particular, if $M_t^i =\lambda I>0$, then 
    \begin{align*}
        \min\{ &\lambda_{\min}(R_t^i + (B_t^{i})^{\top} P_{t + 1}^i B_t^i), \\ &\lambda_{\min}(S_t^i + \lambda I - (B_t^{-i})^{\top} P_{t + 1}^i B_t^{-i})\}
        > \frac{N-1}{2} \lambda.
    \end{align*}
\end{proposition}

\section{Proofs}\label{thm:proof}

\subsection{Proof of~\cref{thm:sre}}

We prove the theorem in two steps.
We show that the game is equivalent to a simultaneous game, which simplifies the analysis. 
We then perform backward induction to compute equilibria of such simultaneous game. 

\subsubsection{Reformulation as a simultaneous game}

We now show that this game is equivalent to a simultaneous game where the cost of each player $i$ is
\begin{multline}\label{eq:cost_augmented}
    J_\mathrm{aug}^i(\pi^i,\sigma^i,\pi^{-i}) = \norm{x_T}_{Q_T^i}^2 \\ + \sum_{t=0}^{T-1} \norm{x_t}^2_{Q_t^i} + \norm{u_t^i}^2_{R_t^i}
    - \norm{d_t^{i}}_{S_t^i}^2
    - 
    \Vert d_t^{i}-u_t^{-i}\Vert^2_{M_t^i},
\end{multline}
subject to the system dynamics and the feedback policies, with its adversary simultaneously maximizing this cost. 
This immediately leads to a revised notion of equilibrium. 

\begin{definition}[equilibrium of the augmented game]
    The policies $(\bar\pi^1,\bar\sigma^{1}, \ldots, \bar\pi^N,\bar\sigma^{N})$ are an equilibrium of the augmented game if for all players $i$ we have 
    \begin{equation}\label{eq:sre:augmented}
        J_\mathrm{aug}^i(\bar \pi^{i}, \sigma^i, \bar\pi^{-i})
        \leq
        J_\mathrm{aug}^i(\bar \pi^{i}, \bar\sigma^i, \bar\pi^{-i})
        \leq
        J_\mathrm{aug}^i(\pi^i, \bar\sigma^i, \bar\pi^{-i})
    \end{equation}
    for all policies $\pi^i\in\Pi^i$ and $\sigma^i\in\Sigma^i$.
\end{definition}

In this formulation, players and adversaries play \emph{simultaneously}. By a standard argument of linear quadratic dynamic zero-sum games~\cite{bacsar2008h}, these two formulations are equivalent. 

\begin{lemma}[equivalence]
    In the setting of~\cref{thm:sre}, the policies $(\bar\pi^1,\bar\sigma^{1}, \ldots, \bar\pi^N,\bar\sigma^{N})$ form an  equilibrium of the augmented game 
    if and only if 
    $(\bar\pi^1,\ldots,\bar\pi^N)$ is a strategically robust dynamic equilibrium, and $(\bar\sigma^1,\ldots,\bar\sigma^N)$ are the worst-case policies in~\eqref{eq:robust_cost} (i.e., the best response to $(\bar\pi^1,\ldots,\bar\pi^N)$). 
\end{lemma}

\begin{proof}
    For fixed $\bar\pi^{-i}$, the theory of dynamic zero-sum games~\cite{bacsar2008h} shows that~\eqref{eq:sre} and~\eqref{eq:sre:augmented} are equivalent.
\end{proof}

\subsubsection{Backward induction}

We compute equilibria of the simultaneous game using backward induction. 
However, distinct players solve Bellman equations with distinct dynamics, and, thus, we are not operating in the standard setting of dynamic games. 
Thus, we first conduct the backward induction and then prove that the obtained solution forms an equilibrium, and this equilibrium is unique. 

\paragraph*{Step 1}
We compute a candidate equilibrium using backward induction
\begin{multline}\label{eq:proof:backward_dp}
    \!\!\!\!
    V_t^i(x_t) = \min_{u_t^i}\max_{d_t^i} 
    \norm{x_t}^2_{Q_t^i} + \norm{u_t^i}^2_{R_t^i} - \norm{d_t^i}^2_{S_t^i}
    -\norm{d_t^i-u_t^{-i}}_{M_t^i}^2
    \\
    +
    V_{t+1}^i\left(A_tx_t + B_t^iu_t^i + B_t^{-i} d_t^i\right).
\end{multline}
We claim that the value functions are $V_t^i(x_t)=x_t^\top P_t^i x_t$, with $P_t^i$ as in~\eqref{eq:sre_riccati}.
We prove this via induction. The base case is trivial. Suppose $V_{t+1}^i(x) =x^\top P_{t+1}^ix$. Then, at stage $t$, backward induction~\eqref{eq:proof:backward_dp} gives
\begin{align*}
    0
    &=
    R_t^iu_t^i + (B_t^i)^\top P_{t+1}^i\left(A_tx_t + B_t^iu_t^i + B_t^{-i} d_t^i\right)
    \\
    0
    &=
    -S^i_td_t^i - M_t^i(d^i_t-u_t^{-i}) 
    \\
    &\qquad\qquad 
    + (B_t^{-i})^\top P^i_{t+1}\left(A_tx_t + B_t^iu_t^i + B_t^{-i} d_t^i\right),
\end{align*} 
where we used the assumptions
$R_{t}^i+(B_t^i)^\top P_{t+1}^iB_t^i\succ 0$ and 
$S_t^i+M_t^i-(B_t^{-i})^\top P^i_{t+1}B_t^{-i}\succ 0$ to
ensure strong convexity and concavity (and thus that min and max can be interchanged, and lower and upper cost-to-go coincide).
Stacking these conditions for all players gives the necessary and sufficient conditions for an equilibrium.
Since all conditions are linear in $(u_t^i,d_t^i,x_t)$, the saddle point $u_t^i$ and $d_t^i$ are a linear function of $x_t$, and $V^i_t(x_t)=x_t^\top P^i_t x_t$ is quadratic. 
This yields~\eqref{eq:sre_riccati} and~\eqref{eq:sre_policies}, and proves the induction. 
We call this candidate equilibrium $(\bar\pi^1,\bar\sigma^1,\ldots,\bar\pi^N,\bar\sigma^N)$.

\paragraph*{Step 2}
We now show that the candidate equilibrium $(\bar\pi^1,\bar\sigma^1,\ldots,\bar\pi^N,\bar\sigma^N)$ from Step~1 satisfies the saddle-point condition~\eqref{eq:sre:augmented}; note that this step would be more demanding in the sequential formulation of the game. 

For the right inequality (optimality of $\bar\pi^i$), fix $d_t^i = L_t^i x_t$ and $u_t^j = K_t^j x_t$ for all $j\neq i$.
Player~$i$ then faces a standard LQR with dynamics
$
    x_{t+1} = \tilde A_t^i x_t + B_t^i u_t^i
$
where 
$
    \tilde A_t^i \coloneqq A_t + B_t^{-i} L_t^i,
$
and cost matrices $\tilde R_t^i = R_t^i$, $\tilde Q_T^i = Q_T^i$, and
$
    \tilde Q_t^i 
    = Q_t^i 
    - (L_t^i)^\top S_t^i L_t^i 
    - (L_t^i - K_t^{-i})^\top M_t^i (L_t^i - K_t^{-i}).
$
The LQR gain for this reduced problem is 
$
    \tilde K_t^i 
    = 
    -\bigl(R_t^i + (B_t^i)^\top \tilde P_{t+1}^i B_t^i\bigr)^{-1} (B_t^i)^\top \tilde P_{t+1}^i \tilde A_t^i.
$
We claim $\tilde K_t^i = K_t^i$ and $\tilde P_t^i = P_t^i$ for all $t$, which we verify by backward induction.
The base case $\tilde P_T^i = Q_T^i = P_T^i$ is immediate.
Suppose $\tilde P_{t+1}^i = P_{t+1}^i$. 
The $u^i$-row of the first-order conditions in Step~1 reads
$H_t^{u^i,u^i} K_t^i+H_t^{u^i,d^i}L_t^i =G_t^{u^i}$, which, by direct inspection, 
gives
$
    K_t^i = \tilde K_t^i.
$
Since the closed-loop matrices coincide, $\tilde A_t^i + B_t^i \tilde K_t^i = A_t + B_t^i K_t^i + B_t^{-i} L_t^i$, the LQR Riccati propagates $\tilde P_t^i = P_t^i$.
Moreover, the reduced LQR is well-posed at every stage because $R_t^i + (B_t^i)^\top P_{t+1}^i B_t^i = H_t^{u^i,u^i} \succ 0$ by assumption.
Thus, $\bar\pi^i$ is the unique optimal policy, confirming the right inequality in~\eqref{eq:sre:augmented}.

The proof of the left inequality (optimality of $\bar\sigma^i$) is analogous and omitted for brevity.

\paragraph*{Step 3}
For uniqueness, assume, for the sake of contradiction, that there is an alternative equilibrium, and let $t'$ be the last time that at least one policy differs from $(\bar\pi^1,\bar\sigma^1,\ldots, \bar\pi^N,\bar\sigma^N)$ defined above. 
Since all policies agree for $t>t'$, the continuation values at $t'+1$ coincide (and are quadratic), and so the matrices $H_{t'}$ and $G_{t'}$ are identical to those in Step~1.
The strict convexity in $u_{t'}^i$ (from $H_{t'}^{u^i,u^i}\succ 0$) and strict concavity in $d_{t'}^i$ (from $H_{t'}^{d^i,d^i}\succ 0$) guarantee that the first-order conditions are necessary and sufficient for each player's and adversary's best response, for every state~$x_{t'}$.
Stacking these conditions yields a unique solution, namely that in the statement of \cref{thm:sre}, so the equilibrium policies at $t'$ must be linear and coincide with the candidate from Step~1. This is a contradiction.

\subsection{Proof of Proposition~\ref{prop:invertibility}}

We study the symmetric part of $H_t$ given by 
$$
\frac{1}{2}(H_t + H_t^\top) = \begin{bmatrix} 
U_t^1 & \dots & V_t^{1,N} \\ 
\vdots & \ddots & \vdots \\ 
V_t^{N,1} & \dots & U_t^N 
\end{bmatrix}
$$
with $U_t^i \!=\!\text{diag}(H_t^{u^i,u^i}, H_t^{d^i,d^i})$ and 
$$
V_t^{i,j} = \begin{bmatrix} 0 & -\frac{1}{2} ({M_t^{j,i}})^\top \\ -\frac{1}{2}({M_t^{i,j}}) & 0 \end{bmatrix}, \quad j \neq i.
$$
By construction, the spectral norm of the coupling blocks is $\|V_t^{i,j}\|_2 = \frac{1}{2}\max\{\sigma_{\max}({M_t^{i,j}}), \sigma_{\max}({M_t^{j,i}})\}$. By definition of $H_t^{\bullet}$, the minimum eigenvalue of the diagonal blocks is $\lambda_{\min}(U_t^i) = \min\{\lambda_{\min}(R_t^i + B_t^{i\top} P_{t + 1}^i B_t^i), \lambda_{\min}(S_t^i + M_t^i - B_t^{-i\top} P_{t + 1}^i B_t^{-i})\}$. 
Invoking the Block Gershgorin Theorem \cite{feingold1962block}, $\frac{1}{2}(H_t + H_t^\top)$ is strictly positive definite if the block diagonal dominance 
$
\lambda_{\min}(U_t^i) > \sum_{j \neq i} \|V_t^{i,j}\|_2
$
holds for all $i \in \{1,\ldots,N\}$.
This is our assumption.
So, $H_t$ is invertible.

\section{Experiments}

We now illustrate the effects and benefits of strategic robustness in a collaborative game and in a network game. In our experiments, we computed policies via the recursion in~\cref{thm:sre} and did not encounter any invertibility issues.

\subsection{One-dimensional collaborative game}

Consider two agents who collaborate to steer the state of a scalar system from $x_0=1$ to the target state $x_T=0$, with, for simplicity, $T=3$.
We consider the integrator dynamics $x_{t+1} = x_t + u_t^1 + u_t^2$ and the control cost $x_T^2 + \sum_{t=0}^{T-1} (u_t^i)^2$.
Across all settings, we use time-invariant robustness parameters for both players; i.e., $M_t^i=m^i I \succ 0$.

\begin{table}[h!]
    \centering
    \caption{Cost percentiles under the biased policy of player 2, comparing player 1’s NE and SR feedback controllers.}
    \begin{tabular}{c|ccccc}
        \toprule
         Percentile   & 5\% & 25\% & 50\% & 75\% & 95\% \\
        \midrule 
         NE & 0.04 & 0.32 & 1.35 & 3.82 & 10.80 \\
         SR & 0.32 & 0.52 & 1.26 & 3.01 & 8.12 \\
         \bottomrule
    \end{tabular}
    \label{tab:cost_statistics}
\end{table}

\subsubsection{Robustness to adversarial perturbations}

To test the robustness of the derived policies, we assess the performance of player 1 against informed adversaries and compare the strategically robust policy for various robustness levels $m^1$ (i.e., player 1 is robust and player 2 is not) against the Nash equilibrium policy ($m^i\to\infty$ for both players). 
We model this adversarial behavior as follows: We define a ``rogue'' strategy for player 2, $u^{2,\text{adv}}_t$, which deviates from the Nash equilibrium policy toward player 1's adversary as follows: 
\begin{equation*}
    \pi^{2,\text{adv}}_t(x_t)
    =
    \pi^{2,\text{NE}}_t(x_t) - \text{clip}(\pi^{2,\text{NE}}_t(x_t) - d^{1,2}_{\text{robust}, t}(x_t), -c,c),
\end{equation*}
where $\pi^{2,\text{NE}}_t$ is the Nash equilibrium policy and $d^{1,2}_{\text{robust}, t}$ is the adversarial policy of player 1, which we obtain when solving for the strategically robust dynamic equilibrium for $m^1  = 1.2, m^2 \to \infty$.
The parameter $c>0$ denotes the perturbation budget, representing the maximum magnitude of deviation allowed, over which we ablate in~\cref{fig:adversarial_costs}.
In the absence of deviations, the Nash equilibrium policy (i.e., $m^1\to\infty$) performs best, but the performance quickly deteriorates when the policy of player 2 becomes more adversarial.
Conversely, strategically robust policies are slightly suboptimal in the absence of perturbations, but do not suffer as much when the other player deviates. 
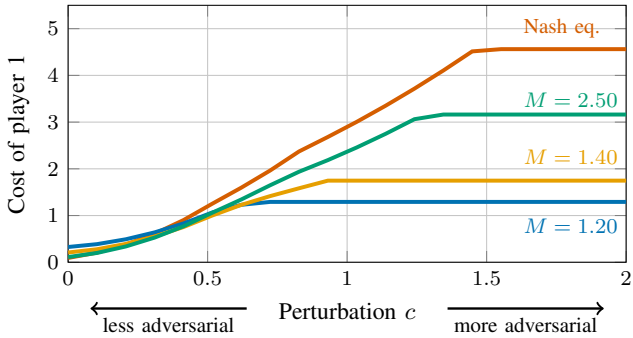
\begin{figure}[!t]
    \centering
    \definecolor{cbBlue}{HTML}{0072B2}
\definecolor{cbOrange}{HTML}{E69F00}
\definecolor{cbGreen}{HTML}{009E73}
\definecolor{cbVermillion}{HTML}{D55E00}

\begin{tikzpicture}
\begin{axis}[
    width=8.1cm,
    height=4.5cm,
    xlabel={Perturbation $c$},
    ylabel={Cost of player 1},
    xmin=0, xmax=2,
    ymin=0, ymax=5.5,
    xtick={0,0.5,1,1.5,2},
    ytick={0,1,2,3,4,5},
    clip=false,
    font=\small,
    grid=major,
    grid style={gray!45},
    tick label style={font=\footnotesize},
    label style={font=\small},
]

\addplot[
    color=cbVermillion,
    line width=1.5pt
] table [col sep=comma, x=clip, y={M=100.000_cost}, restrict x to domain*=0:2] {results/exp2_clipped.csv};

\addplot[
    color=cbBlue,
    line width=1.5pt
] table [col sep=comma, x=clip, y={M=1.200_cost}, restrict x to domain*=0:2] {results/exp2_clipped.csv};

\addplot[
    color=cbOrange,
    line width=1.5pt
] table [col sep=comma, x=clip, y={M=1.400_cost}, restrict x to domain*=0:2] {results/exp2_clipped.csv};

\addplot[
    color=cbGreen,
    line width=1.5pt
] table [col sep=comma, x=clip, y={M=2.500_cost}, restrict x to domain*=0:2] {results/exp2_clipped.csv};

\node[anchor=south west, font=\footnotesize, text=cbVermillion] at (axis cs:1.6,4.55) {Nash eq.};
\node[anchor=south west, font=\footnotesize, text=cbGreen] at (axis cs:1.53,3.12) {$m^1=2.50$};
\node[anchor=south west, font=\footnotesize, text=cbOrange] at (axis cs:1.53,1.9) {$m^1=1.40$};
\node[anchor=north west, font=\footnotesize, text=cbBlue] at (axis cs:1.53,1.14) {$m^1=1.20$};

\draw[->, line width=0.9pt]
    (axis description cs:0.32,-0.18) -- (axis description cs:0.04,-0.18);
\node[anchor=north, font=\footnotesize]
    at (axis description cs:0.18,-0.18) {less adversarial};

\draw[->, line width=0.9pt]
    (axis description cs:0.68,-0.18) -- (axis description cs:0.96,-0.18);
\node[anchor=north, font=\footnotesize]
    at (axis description cs:0.82,-0.18) {more adversarial};

\end{axis}
\end{tikzpicture}
    \vspace{-0.5\baselineskip}
    \caption{Realized cost $J^1_{\mathrm{real}}$ experienced by player 1 under varying adversarial perturbation budgets $c$. Strategically robust policies demonstrate better robustness compared to the Nash equilibrium policy, at the price of modest suboptimality when the perturbation is small.}
    \label{fig:adversarial_costs}
    \vspace{-1ex}
\end{figure}

\subsubsection{Robustness to random perturbations}
We now investigate the performance of strategically robust feedback policies against randomly perturbed policies. Specifically, we suppose the feedback policy of player 2 is corrupted by a constant, randomly sampled drift; i.e., 
$
    \pi_t^{2, \text{bias}}(\cdot) = \pi_t^{2, \text{NE}}(\cdot) + b
$
with $b \sim \mathcal{N}(0, 1)$.
Through Monte Carlo simulations with $10^5$ samples, we compare in~\cref{tab:cost_statistics} the cost distributions of the Nash equilibrium policies and the strategically robust policies when player 1 uses $m^1=1.2$ (and $m^2\to+\infty$).
Strategic robustness significantly improves the tail of the cost distribution. We observe a 25\% cost reduction in the 95th percentile, 21\% cost reduction in the 75th percentile, and 7\% improvement in the median, with almost no change in the lower percentiles.

\subsubsection{Robustness induces collaboration}
In our last experiment, we investigate the effect of strategic robustness on the social cost, defined as the sum of all agents' costs. 
We do so in two settings, one where we use the same level of robustness for both players $M^i=m I \succ 0$ and one where we use $m^1=m$ and set $m^2\to\infty$ (i.e., player 2 is not robust), and report results in \cref{fig:collaborative:freelunch}.
Remarkably, we observe that social cost slightly decreases as agents become more strategically robust---a ``free-lunch'' effect already observed in static games in~\cite{lanzetti2025strategically} (there referred to as ``coordination via robustification'') and used in~\cite{qu2026training} for training collaborative AI agents in multi-agent reinforcement learning. 
This effect is purely game-theoretic and contrasts with the intuition on the effect of robustness in optimization. Indeed, while in robust optimization the decision-maker regrets taking robust decisions if the nominal scenario realizes, strategic robustness in games can lead to benefits for all players and better social cost---even in the absence of perturbations from the equilibrium.
In this example, we can explain this phenomenon as follows. To protect against misspecified policies of other agents (i.e., the others exerting less control effort), strategic robustness induces players to exert a larger control effort,  which steers the system closer to the target state and improves social cost.

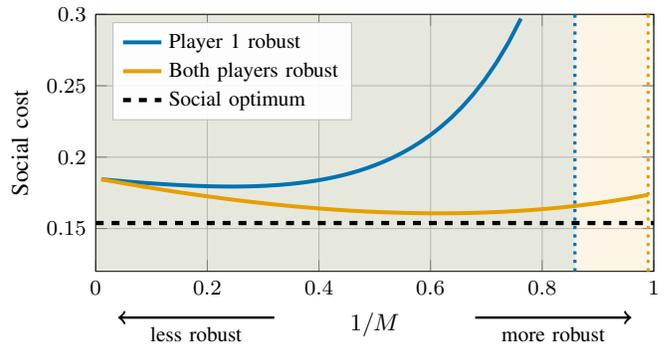
\begin{figure}[!t]
    \centering
    \definecolor{cbBlue}{HTML}{0072B2}
\definecolor{cbOrange}{HTML}{E69F00}
\definecolor{cbGreen}{HTML}{009E73}
\definecolor{cbVermillion}{HTML}{D55E00}

\def\data{results/cost_vs_M_2p_asymmetric.csv}
\def\datasym{results/cost_vs_M_2p_symmetric.csv}
\def\xfeasblue{0.858738}
\def\xfeasorange{0.990099}
\def\xstopblue{0.761982}
\def\xstoporange{0.990099}

\begin{tikzpicture}
\begin{axis}[
    width=8.1cm,
    height=4.5cm,
    xlabel={$1/m$},
    ylabel={Social cost},
    xmin=0.0, 
    xmax=1.0,
    ymax=0.3,
    ymin=0.12, 
    clip=false,
    grid=major,
    grid style={gray!45},
    legend style={at={(0.03,0.97)}, anchor=north west, font=\footnotesize, draw=gray!50, fill=white, fill opacity=0.92, text opacity=1},
    legend cell align={left},
    legend image post style={xscale=0.85},
    tick label style={font=\footnotesize},
    label style={font=\small},
]

\path[fill=cbBlue, fill opacity=0.10, draw=none]
    (axis cs:0,\pgfkeysvalueof{/pgfplots/ymin}) rectangle
    (axis cs:\xfeasblue,\pgfkeysvalueof{/pgfplots/ymax});
\path[fill=cbOrange, fill opacity=0.10, draw=none]
    (axis cs:0,\pgfkeysvalueof{/pgfplots/ymin}) rectangle
    (axis cs:\xfeasorange,\pgfkeysvalueof{/pgfplots/ymax});

\addplot[color=cbBlue, dotted, line width=1.2pt, forget plot]
    coordinates {(\xfeasblue,\pgfkeysvalueof{/pgfplots/ymin}) (\xfeasblue,\pgfkeysvalueof{/pgfplots/ymax})};
\addplot[color=cbOrange, dotted, line width=1.2pt, forget plot]
    coordinates {(\xfeasorange,\pgfkeysvalueof{/pgfplots/ymin}) (\xfeasorange,\pgfkeysvalueof{/pgfplots/ymax})};

\addplot[color=cbBlue, line width=1.5pt]
    table [col sep=comma, x={1/M}, y=TotalCost, restrict expr to domain={\thisrow{1/M}}{0:\xstopblue}] {\data};
\addlegendentry{Player 1 robust}

\addplot[color=cbOrange, line width=1.5pt]
    table [col sep=comma, x={1/M}, y=TotalCost, restrict expr to domain={\thisrow{1/M}}{0:\xstoporange}] {\datasym};
\addlegendentry{Both players robust}

\addplot[black, dashed, line width=1.6pt]
    coordinates {(\pgfkeysvalueof{/pgfplots/xmin},0.153846) (\pgfkeysvalueof{/pgfplots/xmax},0.153846)};
\addlegendentry{Social optimum}

\draw[->, line width=0.9pt]
    (axis description cs:0.32,-0.18) -- (axis description cs:0.04,-0.18);
\node[anchor=north, font=\footnotesize]
    at (axis description cs:0.18,-0.18) {less robust};

\draw[->, line width=0.9pt]
    (axis description cs:0.68,-0.18) -- (axis description cs:0.96,-0.18);
\node[anchor=north, font=\footnotesize]
    at (axis description cs:0.82,-0.18) {more robust};

\end{axis}
\end{tikzpicture}
    \vspace{-.5\baselineskip}
    \caption{Social cost, defined as the sum of the agents' realized costs, as a function of the level of robustness $m$ when only player 1 is strategically robust (blue) and all players are strategically robust with the same $m$ (yellow). The shaded area is the range of $m$ for which the strategically robust equilibrium exists, before cost diverges to $+\infty$ as the adversary carries too much power.}
    \label{fig:collaborative:freelunch}
    \vspace{-1ex}
\end{figure}

\subsection{Network consensus and control}

In our second experiment, we consider a multi-agent control problem on a graph, as proposed in \cite{pasqualetti2014controllability}.
Consider a network of $N=5$ nodes arranged in a star configuration as shown in~\cref{fig:network:star}.
The system state of node $i$ evolves according to the dynamics 
$
    x_{t+1}^i
    = 
    \frac{2}{3} x_t^i
    +
    \frac{1}{3\abs{\mathcal N^i}} \sum_{j\in \mathcal N^i}x_t^j
    +
    u_t^i, 
$
where $\mathcal N^i$ is the set of neighbors of node $i$ and the initial condition is $x_0^i=0$.
Each agent $i$ aims to minimize control effort while reaching the target state $x_\mathrm{target}=5$, so that 
$
    J^i_{\mathrm{real}}(\pi^i,\pi^{-i}) = \sum_{t=0}^{T-1} (u_t^i)^2 + (x^i_{T} - 5)^2, 
$
over a horizon of $T=20$. While our formulation does not include an affine term in the cost, it can be modeled via a standard state augmentation. 

\subsubsection{Central node resilience under leaf perturbations}

We examine the system's resilience by comparing the Nash equilibrium feedback policy ($m^i \to \infty$ for all $i$) and a strategically robust configuration where the central node (node 1) is strategically robust (with $m^1= 0.16$).
To assess resilience, we consider the adversarial policy
\begin{equation}\label{eq:network:adversary}
    \pi^{\mathrm{leaf}\,j,\text{adv}}_t(x_t)
    =
    \begin{cases}
        d^{\mathrm{center},\mathrm{leaf}\,j}_{t}(x_t) & j\in\{1,2\},
        \\
        \pi_t^{\mathrm{NE},\mathrm{leaf}\,j}(x_t) & j\notin\{1,2\},
    \end{cases}
\end{equation}
where $d^{\mathrm{center},\mathrm{leaf}\,j}_{t}$ is the policy of the adversary of the center node. 
We evaluate the performance of the central node across four scenarios:
(i) Nash equilibrium, 
(ii) strategically robust dynamic equilibrium, 
(iii) Nash equilibrium center and adversarial leaves 1 and 2 as in~\eqref{eq:network:adversary}, 
(iv) strategically robust center node and adversarial leaves 1 and 2 as in~\eqref{eq:network:adversary}.
We measure the terminal state $x^1_T$ and the total realized cost $J^1_{\mathrm{real}}$. As shown in~\cref{tab:results} and \cref{fig:star_trajectories}, the strategically robust feedback policy demonstrates better resilience, with a state trajectory that suffers much less from the adversary. 

\begin{table}[!t]
\centering
\begin{minipage}{0.30\columnwidth}
\centering
\vspace{-0.25\baselineskip}
\begin{tikzpicture}[
    scale=0.85,
    every node/.style={circle, draw, fill=gray!15, minimum size=6.5pt, inner sep=1.8pt},
    every path/.style={line width=1.1pt}
]
    \node (c) at (0,0) {};
    \node (n2) at (90:1.35) {};
    \node (n3) at (0:1.35) {};
    \node (n4) at (-90:1.35) {};
    \node (n5) at (180:1.35) {};
    \draw (c) -- (n2);
    \draw (c) -- (n3);
    \draw (c) -- (n4);
    \draw (c) -- (n5);
\end{tikzpicture}

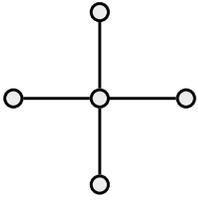
\captionof{figure}{Star graph with $N=5$ nodes.}\label{fig:network:star}
\vspace{-0.8\baselineskip}
\end{minipage}
\hfill 
\begin{minipage}{0.66\columnwidth}
\centering
\caption{Central Node Performance Under Adversarial Perturbation}
\vspace{-1ex}
\label{tab:results}
\begin{tabular}{lcc}
\toprule
Scenario & \shortstack{Terminal state\\$x_T^1$} & \shortstack{Central node\\ realized cost} \\
\midrule
NE & 4.32 & 2.39 \\
SR & 4.77 & 6.23 \\
NE (adv.) & 2.24 & 26.00 \\
SR (adv.) & 4.04 & 14.47 \\
\bottomrule
\end{tabular}
\end{minipage}
\end{table}

\begin{figure}[!t]
    \centering
    \definecolor{cbBlue}{HTML}{0072B2}
\definecolor{cbOrange}{HTML}{E69F00}
\definecolor{cbGreen}{HTML}{009E73}
\definecolor{cbVermillion}{HTML}{D55E00}

\begin{tikzpicture}
\begin{axis}[
    width=7.7cm,
    height=4.5cm,
    xlabel={Time},
    ylabel={State (central node)},
    ylabel style={yshift=-6pt},
    xmin=0, xmax=20,
    ymin=-3.6, ymax=6,
    xtick={0,5,10,15,20},
    grid=major,
    grid style={gray!45},
    legend columns=-1,
    legend style={at={(0.5,1.04)}, anchor=south, font=\footnotesize, draw=gray!60, fill=white},
    legend image post style={xscale=0.85},
    legend cell align={left},
    tick label style={font=\footnotesize},
    label style={font=\small},
]

\addplot[black, dashed, line width=1.6pt] coordinates {(0,5) (20,5)};
\addlegendentry{Target}
\addplot[color=cbGreen, line width=1.5pt] table [col sep=comma, x=time, y={A: All Neutral}] {results/centralized_center_state.csv};
\addlegendentry{NE}
\addplot[color=cbVermillion, line width=1.5pt] table [col sep=comma, x=time, y={B: Robust Center}] {results/centralized_center_state.csv};
\addlegendentry{SR}
\addplot[color=cbBlue, line width=1.5pt] table [col sep=comma, x=time, y={C: Neutral + Disc}] {results/centralized_center_state.csv};
\addlegendentry{NE (adv.)}
\addplot[color=cbOrange, line width=1.5pt] table [col sep=comma, x=time, y={D: Robust + Disc}] {results/centralized_center_state.csv};
\addlegendentry{SR (adv.)}
\end{axis}
\end{tikzpicture}
    \vspace{-0.5\baselineskip}
    \caption{State of the central node at the Nash equilibrium (NE), the strategically robust dynamic equilibrium (SR), and in the presence of adversarial perturbations of the leaf nodes. The performance of the Nash equilibrium policy significantly deteriorates with adversarial perturbations, whereas the strategically robust policy continues to perform well. 
    }
    \label{fig:star_trajectories}
    \vspace{-1ex}
\end{figure}
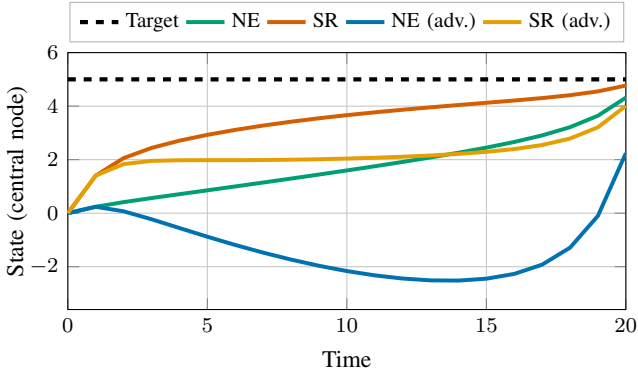

\subsubsection{Robustness induces coordination}
We now investigate the impact of strategic robustness on social cost, defined as the sum of the realized costs of all players. 
We again consider two settings, one in which Player 1 is the only strategically robust agent (i.e., $M^1=m I$ and $m^i\to\infty$ for $i\neq 1$) and one in which all players have the same robustness levels.
As we show in~\cref{fig:network:free-lunch}, we observe also here a ``free-lunch'' effect whereby social cost decreases for some level of robustness.

\begin{figure}[!t]
    \centering
    \definecolor{cbBlue}{HTML}{0072B2}
\definecolor{cbOrange}{HTML}{E69F00}
\definecolor{cbGreen}{HTML}{009E73}
\definecolor{cbVermillion}{HTML}{D55E00}

\def\data{results/cost_vs_M_cent_asymmetric.csv}
\def\datasym{results/cost_vs_M_cent_symmetric.csv}
\def\xfeasblue{7.142857}
\def\xfeasorange{1.708277}
\def\xstopblue{8.386554}
\def\xstoporange{1.700331}

\begin{tikzpicture}
\begin{axis}[
    width=7.7cm,
    height=4.5cm,
    xlabel={$1/m$},
    ylabel={Social cost},
    xmin=0.0,
    xmax=8,
    ymin=3.8,
    ymax=9,
    clip mode=individual, %
    grid=major,
    grid style={gray!45},
    legend style={at={(0.03,0.03)}, anchor=south west, font=\footnotesize, draw=gray!50, fill=white, fill opacity=0.92, text opacity=1},
    legend cell align={left},
    legend image post style={xscale=0.85},
    tick label style={font=\footnotesize},
    label style={font=\small},
]

\path[fill=cbBlue, fill opacity=0.10, draw=none]
    (axis cs:0,\pgfkeysvalueof{/pgfplots/ymin}) rectangle
    (axis cs:\xfeasblue,\pgfkeysvalueof{/pgfplots/ymax});
\path[fill=cbOrange, fill opacity=0.10, draw=none]
    (axis cs:0,\pgfkeysvalueof{/pgfplots/ymin}) rectangle
    (axis cs:\xfeasorange,\pgfkeysvalueof{/pgfplots/ymax});

\addplot[color=cbBlue, dotted, line width=1.2pt, forget plot]
    coordinates {(\xfeasblue,\pgfkeysvalueof{/pgfplots/ymin}) (\xfeasblue,\pgfkeysvalueof{/pgfplots/ymax})};
\addplot[color=cbOrange, dotted, line width=1.2pt, forget plot]
    coordinates {(\xfeasorange,\pgfkeysvalueof{/pgfplots/ymin}) (\xfeasorange,\pgfkeysvalueof{/pgfplots/ymax})};

\addplot[color=cbBlue, line width=1.5pt]
    table [col sep=comma, x={1/M}, y=TotalCost, restrict expr to domain={\thisrow{1/M}}{0:\xstopblue}] {\data};
\addlegendentry{Player 1 robust}

\addplot[color=cbOrange, line width=1.5pt]
    table [col sep=comma, x={1/M}, y=TotalCost, restrict expr to domain={\thisrow{1/M}}{0:\xstoporange}] {\datasym};
\addlegendentry{Both players robust}

\addplot[black, dashed, line width=1.6pt]
    coordinates {(\pgfkeysvalueof{/pgfplots/xmin},4.011799) (\pgfkeysvalueof{/pgfplots/xmax},4.011799)};
\addlegendentry{Social optimum}

\draw[->, line width=0.9pt]
    (axis description cs:0.32,-0.18) -- (axis description cs:0.04,-0.18);
\node[anchor=north, font=\footnotesize]
    at (axis description cs:0.18,-0.18) {less robust};

\draw[->, line width=0.9pt]
    (axis description cs:0.68,-0.18) -- (axis description cs:0.96,-0.18);
\node[anchor=north, font=\footnotesize]
    at (axis description cs:0.82,-0.18) {more robust};

\end{axis}
\end{tikzpicture}
    \vspace{-0.5\baselineskip}
    \caption{Social cost, defined as the sum of the agents' costs, as a function of the level of robustness $m$ when only player 1 is strategically robust (blue) and all players are strategically robust with the same $m$ (yellow). The shaded area is the range of $m$ for which the strategically robust equilibrium exists, before cost diverges to $+\infty$ as the adversary carries too much power.
    }
    \label{fig:network:free-lunch}
    \vspace{-1ex}
\end{figure}
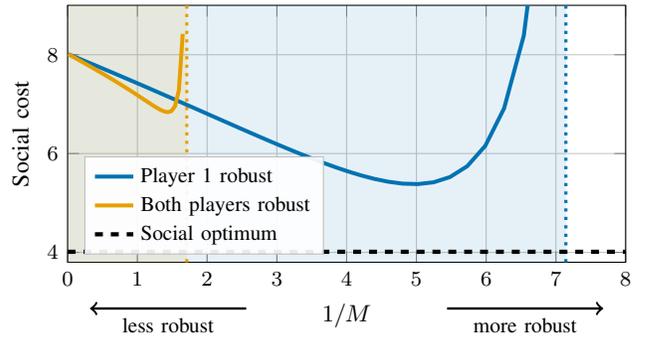

\bibliographystyle{IEEEtran}
\bibliography{references}

\end{document}